\documentclass{amsart}

\usepackage[pdftex]{graphicx}
\usepackage{float}
\usepackage{hyperref}
\usepackage{amssymb}
\usepackage{amsmath}
\usepackage{amsthm}

\newcommand{\A}{A_n}
\newcommand{\F}{\mathcal F}
\newcommand{\G}{\mathcal G}
\newcommand{\m}{\mathbf b}
\newcommand{\rep}{\operatorname{rep}}
\newcommand{\Hom}{\operatorname{Hom}}
\newcommand{\Ext}{\operatorname{Ext}}
\newcommand{\N}{\mathbf M}
\newcommand{\n}{\mathbf a}
\newcommand{\C}{\mathcal C}
\renewcommand{\split}{s}
\renewcommand{\iota}{f}
\newcommand{\D}{\mathcal D}
\newcommand{\pr}{g}

\newcommand{\Gen}{\operatorname{Gen}}

\theoremstyle{plain}
\newtheorem{theorem}{Theorem}[section]
\newtheorem{proposition}[theorem]{Proposition}
\newtheorem{lemma}[theorem]{Lemma}
\newtheorem{corollary}[theorem]{Corollary}

\begin{document}

\title{The Tamari lattice as it arises in quiver representations}
\author{Hugh Thomas}
\address{Department of Mathematics and Statistics,
University of New Brunswick, Fredericton, NB, E3B 5A3, Canada.}
 \email{hthomas@unb.ca}

\begin{abstract}In this chapter, we explain how the Tamari lattice arises in 
the context of the representation theory of quivers, as the poset whose
elements are the torsion classes of a directed path quiver, with the order
relation given by inclusion.\end{abstract}

\maketitle

\section{Introduction}

In this chapter, we will explain how the Tamari lattice $T_n$ 
arises in the 
context of the representation theory of quivers.  In the representation theory
of quivers, one fixes a quiver (quiver being a synonym for ``directed graph'')
and then considers the category of {\it representations} of that quiver.  
(Terms which are not defined in this introduction will be defined shortly.)
Subcategories of this category with certain natural properties are called
{\it torsion classes}.  We show that the Tamari lattice $T_n$ arises 
as the set of torsion classes, ordered by inclusion, for the quiver 
consisting of
a directed path of length $n$.  It therefore follows that, for any directed 
graph, we obtain a generalization of the Tamari lattice.  At the end
of this chapter, we will comment
briefly on the lattices that arise in this way, which include the Cambrian
lattices discussed in Reading's contribution to this volume \cite{T_re}.  

The treatment
of quiver representations which we have undertaken is very elementary.  
In particular, we avoid all use of homological algebra.  
A reader familiar with quiver representations will have no
trouble finding quicker proofs of the results we present here.  
Introductions to 
quiver representations from a more algebraically sophisticated point
of view may be found in \cite{T_ARS,T_ASS}.  

\section{Quiver representations}
Let $Q$ be a quiver (i.e., a directed graph).  Fix a ground field $K$.  
A representation of $Q$ is an assignment of a 
finite-dimensional vector space $V_i$ over $K$ 
to each vertex $i$ of $Q$, and a linear map $V_\alpha:V_i\rightarrow V_j$ to each arrow $\alpha:i\rightarrow j$ of $Q$.  

For a pair of representations $V,W$ of $Q$, 
we define a morphism from $V$ to $W$ 
to consist of a collection of maps $f_i:V_i\rightarrow W_i$ for all 
vertices $i$, such
that for any $\alpha:i\rightarrow j$, we have that  $W_\alpha \circ f_i =
f_j \circ V_\alpha$.  We write $\Hom(V,W)$ for the set of morphisms from
$V$ to $W$.  It as a natural $K$-vector space structure.   
As usual, an isomorphism is a morphism which is 
invertible.  An injection is a morphism all of whose linear maps are 
injections; surjections are defined similarly.

These definitions make quiver representations into a {\it category}, 
which we denote $\rep Q$.  
(The careful reader is encouraged to confirm this.)

Given two representations of $Q$, their direct sum
$V\oplus W$ is defined in the obvious way: 
setting $(V\oplus W)_i = V_i \oplus W_i$, and 
$(V\oplus W)_\alpha=V_\alpha\oplus W_\alpha$.  

A representation is called {\it indecomposable} if it is not isomorphic to 
the direct sum of two non-zero representations.  

\section{Subrepresentations, quotient representations, and extensions}

If $Y$ is a representation of $Q$, a {\it subrepresentation} of $Y$ is
a representation $X$ such that for each $i$, $X_i$ is a subspace of 
$Y_i$, and for $\alpha:i\rightarrow j$, we have that $X_\alpha$ 
is induced from the inclusions of $X_i$ and 
$X_j$ into $Y_i$ and $Y_j$ respectively.  The inclusions of $X_i$ into
$Y_i$ define an injective morphism from $X$ to $Y$.  

If $Y$ is a representation of $Q$, and $x\in Y_i$, the subrepresentation
of $Y$ generated by $x$ is the representation $X$ such that 
$X_j$ is spanned by all images of $x$ under linear maps corresponding
to walks from $i$ to $j$ in $Q$.  

If $Y$ is a representation of $Q$, and $X$ is a subrepresentation of $Y$,
then it is also possible to form the {\it quotient representation}
$Y/X$.  By definition $(y/X)_i=Y_i/X_i$, and the maps of $Y/X$ are 
induced from the maps of $Y$.  The quotient maps from $Y_i$ to
$(Y/X)_i$ define a surjective morphism from $Y$ to $Y/X$.

Suppose $X,Y,Z$ are representations of $Q$.  Then $Y$ is said to be an 
{\it extension}
of $Z$ by $X$ if there is a subrepresentation of $Y$ which is isomorphic
to $X$, such that the corresponding quotient representation is isomorphic to
$Z$.  The extension is called {\it trivial} if there is a morphism $\split$ from $Y$
to $X$ which is the identity on $X$.  Such a morphism is said to split
the inclusion of $X$ into $Y$.

\begin{lemma} \label{T_split} If $Y$ is a trivial extension of $Z$ by $X$, then 
$Y$ is isomorphic to $X\oplus Z$.  
\end{lemma}

\begin{proof}
Let $\split$ be the map which splits the inclusion of $X$ into $Y$.  
Write $\pr$ for the
quotient map from $Y$ to $Z$.  
Then $\split\oplus\pr$ is a morphism from $Y$ to $X\oplus Z$, which is 
an isomorphism over each vertex.  It follows that it is an isomorphism 
of representations.  
\end{proof}



The following discussion is not necessary for our present considerations,
but may be of interest, in that it connects our discussion to notions
of homological algebra.  
It is possible to define a notion of equivalence on extensions, as 
follows: two extensions $Y,Y'$ of $X$ by $Z$ are said to be equivalent if 
there is an isomorphism from $Y$ to $Y'$ which induces the identity maps 
on $X$ and $Z$.  
We write 
$\Ext(Z,X)$ for the set of extensions of $Z$ by $X$ up to equivalence.  
This turns out to have a natural $K$-vector space structure.  $\Ext(Z,X)$ can
then be identified as $\Ext^1(Z,X)$ in the usual sense of homological algebra.
See \cite[Appendix A.5]{T_ASS}.  

\section{Pullbacks of extensions}

\begin{lemma} \label{T_pullback}
Let $Y$ be an extension of $Z$ by $X$.  Suppose we have a surjective map 
$h:Z'\rightarrow Z$.  Then there is a representation $Y'$ which is an
extension of $Z'$ by $X$, and such that $Y'$ admits a surjection to $Y$.  
\end{lemma}

The extension which we will exhibit in order to prove this lemma is
called the {\it pullback} of the extension $Y$ along the surjection 
$h$.  

\begin{proof} Let $i$ be a vertex of $Q$.  We are given surjective maps
$\pr:Y_i\rightarrow Z_i$ and $h:Z'_i\rightarrow Z_i$.  Define $Y'_i$ to be the 
pullback of these two maps, that is to say, 
$$Y'_i = \{(y,z')\mid y\in Y, z'\in Z', \textrm{ and } \pr(y)=h(z')\}.$$  

For $\alpha:i\rightarrow j$ an arrow of $Q$, define 
$Y'_\alpha= Y_\alpha\times Z'_\alpha$.  One verifies that this defines a map
from $Y'_i$ to $Y'_j$.  It follows that $Y'$ is a representation of $Q$.  

Write $\iota_i:X_i\rightarrow Y_i$ for the given injection from $X_i$ to 
$Y_i$.  There is an injective morphism $\iota'$ from 
$X$ to $Y'$, defined on $X_i$ by sending $x$ to $(\iota(x),0)$.  One 
checks that $(\iota(x),0)\in Y'_i$, and that these maps define a morphism
from $X$ to $Y'$.  

It is also easy to see that there is a surjective morphism from $Y'$ to 
$Z'$ defined on $Y_i'$ by sending $(y,z')$ to $z'\in Z_i'$.  The elements
of $Y_i'$ which are sent to zero by this map are those of the form
$(y,0)$, and $(y,0)\in Y_i'$ iff $\pr(y)=0$ iff $(y,0)\in \iota'_i(X_i)$.  So 
$Z'$ is isomorphic to $Y'/X$, as desired.  

Finally, one 
defines a map from $Y'_i$ to $Y_i$ by sending $(y,z')\rightarrow y$.
One checks that this defines a morphism from $Y'$ to $Y$, which is clearly
surjective.  
\end{proof}
 
\section{Indecomposable representations of the quiver $\A$}
Consider the quiver which consists of an oriented path: the vertices are 
numbered 1 to $n$, and for $1\leq i \leq n-1$, there is a unique arrow
$\alpha_i$ 
whose tail is at vertex $i$, and whose head is at vertex $i+1$.  We will
refer to this quiver as $\A$.





For $1\leq i\leq j \leq n$, define a representation $E^{ij}$ by
putting one-dimensional vector spaces at all vertices $p$ with 
$i\leq p\leq j$, with identity maps between 
successive one-dimensional vector spaces, and zero vector spaces and 
maps elsewhere.  

\begin{proposition} \label{T_indecomposables}
The representations $E^{ij}$ are indecomposable, and any 
indecomposable
representation of $\A$ is isomorphic to some $E^{ij}$.
\end{proposition}

\begin{proof} Suppose $E^{ij}\cong X\oplus Y$.  Since for $E^{ij}$ the vector spaces 
at each vertex are at most one-dimensional, for a given vertex $p$, at most
one of $X_p$ and $Y_p$ is non-zero.  If neither $X$ nor $Y$ is zero, 
there must be some $i\leq p < j$ 
such that either $X_p$ is zero and $Y_{p+1}$ is zero, or vice versa.  
In either case, $(X\oplus Y)_{\alpha_p} = 0$.   However,
 $(E^{ij})_{\alpha_p} \ne 0$,
and it follows 
that $E^{ij}$ is not isomorphic to $X\oplus Y$.  

Let $V$ be an indecomposable representation of $\A$. Write $p_j$ for 
$V_{\alpha_j}$.   
Choose $i$ minimal such that $V_i\ne 0$, and choose a non-zero $t\in V_i$.
Let $T$ be the subrepresentation of $V$ generated by $t$, which admits
a natural injection into $V$.  We have natural inclusions of the 
vector space at vertex $k$ for $T$ into the vector space at vertex $k$ for
$V$, and we denote this inclusion by $\iota_k:T_k\rightarrow V_k$.

$$\includegraphics{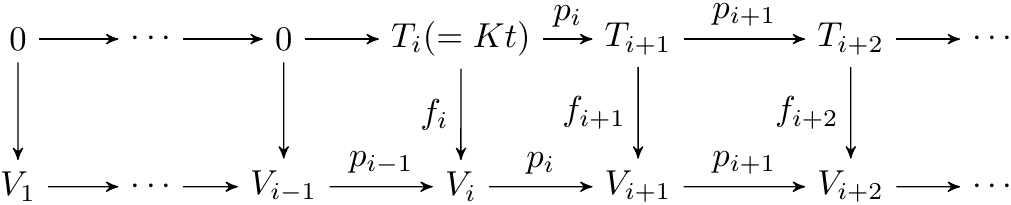}$$


Let $j$ be maximal so that $p_{j-1}\dots p_i(t)\ne 0$.  Define a map 
$\split_j$ which splits the inclusion $\iota_j$, that is to say, a map
such that $\split_j\circ \iota_j$ is the identity.  Now inductively
define $\split_{j-1}, \split_{j-2},\dots, \split_i$ so that, when constructing
$\split_k$, we have that $\split_k$ splits $\iota_k$, and 
$p_k \circ \split_k = \split_{k+1} \circ p_{k}$.  For $k$ not between $i$ 
and $j$, define $\split_k=0$.  

We claim the maps $\split_k$ define a morphism from $V$ to $T$.  
The only conditions which we did not explicitly build into the construction
of $\split_k$ are the commutativity conditions 
$p_{i-1}\circ \split_{i-1} = \split_i\circ p_{i-1}$  (if $i>1$) and
$p_{j}\circ \split_j = \split_{j+1}\circ p_j$ (if $j<n$).  The first 
is satisfied by our assumption that $i$ is minimal such that $V_i\ne 0$, 
and the second is satisfied by our assumption that $p_j\dots p_i(x)=0$,
which implies that $p_{j}|_{T_{j}}=0$.

By Lemma \ref{T_split}, it follows that $V$ is isomorphic to the direct sum of 
$T$ and $V/T$.  Since $V$ is indecomposable by assumption, and $T$ is
non-zero, $V/T$ must be zero, so $V$ is isomorphic to $T$, which is 
isomorphic to $E^{ij}$, proving the proposition.\end{proof}

\section{Morphisms and extensions between indecomposable representations of $\A$}

\begin{proposition}\label{T_morphisms} The space of morphisms from $E^{ij}$ to $E^{kl}$ is either
0-dimensional or 1-dimensional.  It is one-dimensional iff 
$k\leq i \leq l \leq j$.   
\end{proposition}

\begin{proof} 
$E^{ij}$ is generated by $(E^{ij})_i$, so a morphism $f:E^{ij}\rightarrow
E^{kl}$ is determined by its restriction to $(E^{ij})_i$.  The space of 
maps from $(E^{ij})_i$ to $(E^{kl})_i$ is one-dimensional if $k\leq i \leq l$,
and zero otherwise.  If $l>j$, then the commutativity condition
corresponding to $\alpha_j$ cannot be satisfied for a non-zero morphism;
if on the other hand $l\leq j$, then we see that non-zero morphisms do exist. 
\end{proof}

\begin{proposition}\label{T_extensions} The only circumstance in which 
there is a non-trivial extension of $E^{ij}$ by $E^{kl}$ is 
if $i+1 \leq k\leq j+1 \leq l$.  In this case, any non-trivial extension
is isomorphic to $E^{il}\oplus E^{kj}$.  (If $k=j+1$, we interpret 
$E^{kj}$ as zero.)
\end{proposition}

\begin{proof} Let $Y$ be an extension of $E^{ij}$ by $E^{kl}$.  Let 
$t$ be an element of $Y_i$ which maps to a non-zero element of 
$(E^{ij})_i$.  Let $T$ be the sub-representation of
$Y$ which is generated by $t$.  The representation 
$T$ is definitely non-zero at the
vertices $p$ with $i\leq p \leq j$, since the image of $t$ in 
$(E^{ij})_p$ is non-zero for such $p$.  If $T_{j+1}=0$, then       
$T$ is isomorphic to $E^{ij}$, and
the projection from $Y$ to $E^{ij}$ splits the inclusion of $T$ into $Y$, so $Y$ is isomorphic to $E^{ij}\oplus E^{kl}$ by Lemma \ref{T_split}.  

Therefore, in order for there to exist a non-trivial extension, we must
definitely have $k\leq j+1 \leq l$.  Consider now the case that $k\leq i$,
which we must also exclude.  Let $v$ be the image of $t$ in
$Y_{j+1}$, and suppose it is non-zero.  Since $E^{ij}$ is not supported 
over $j+1$, we must have that $v$ lies in the image of $E^{kl}$.  
By our assumption that $k\leq i$, there is an element $x$ of 
$(E^{kl})_i$ such that its image in $(E^{kl})_{j+1}$ coincides with $v$.  
Now set $t'=t-v$, and repeat the above analysis with $t'$.  By construction
the image of $t'$ in $Y_{j+1}$ is zero, so $E^{ij}$ is a direct summand of
$Y$ and the extension is trivial.  

Finally, suppose that $i+1\leq k \leq j+1 \leq l$, and that we have some 
$t$ in $Y_i$ whose image in $Y_{j+1}$ is non-zero.  It follows necessarily
that the subrepresentation $T$ of $Y$ generated by $t$ must be isomorphic to
$E^{il}$.   As 
in the proof of Proposition \ref{T_indecomposables}, we see that $T$ is 
a direct summand of $Y$, so $Y$ is isomorphic to $E^{il}\oplus Z$ for
some $Z$, and it is clear that we must have $Z\cong E^{kj}$ (or $Z=0$ if
$k=j+1$).
\end{proof}

The only nontrivial extension of indecomposable representations of $A_2$ 
is an extension of $E^{11}$ by $E^{22}$, isomorphic to $E^{12}$.  There are 
several examples of nontrivial extensions among representations of 
$A_3$, such as, for example, the extension of $E^{12}$ by $E^{23}$ which
is isomorphic to $E^{13}\oplus E^{22}$.  Note that this latter example is 
obviously non-trivial even though the extension is a direct sum, since
the summands in the direct sum are not the same as the two indecomposables
from which the extension was built.

\section{Subcategories of $\rep Q$}
By definition, a subcategory of a category is a category whose objects and morphisms belong to those of the 
original category, and such that the identity maps in the subcategory and 
category coincide.  This notion is not strong enough for our purposes.  
A full additive subcategory $\mathcal B$ of $\rep Q$ is a subcategory satisfying the following conditions:
\begin{itemize}
\item For $X,Y$ objects of $\mathcal B$, we have that
$\Hom_{\mathcal B}(X,Y)=\Hom(X,Y)$. 
\item There is some set of indecomposable objects of $\rep Q$ such that
the objects of $\mathcal B$ consist of all finite direct sums of indecomposable
objects from this set.
\end{itemize}
From now on, when we speak of subcategories, we always mean full additive
subcategories.  

In this chapter, 
we are particularly interested in {\it torsion classes} in $\rep Q$.  
A 
torsion class in an abelian category $\mathcal A$ is a full additive subcategory
closed under quotients and extensions.  That is to say, if $Y\in \mathcal A$,
and there is a surjection from $Y$ to $Z$, then $Z\in \mathcal A$, and if
$X,Z$ are in $\mathcal A$, and $Y$ is an extension of $Z$ by $X$, then
$Y\in \mathcal A$.  Torsion classes play an important role in tilting theory,
which it is beyond the scope of this chapter to review.  See \cite{T_ARS,T_ASS}
for more information on this subject.  

\section{Quotient-closed subcategories}
As a prelude to classifying the torsion classes of $\rep \A$, we consider
the subcategories of $\rep \A$ which are closed under quotients.  
This class of subcategories of $\rep Q$ 
was not studied classically but 
has received some recent attention (see \cite{T_ort} or, 
considering the equivalent dual case, \cite{T_ri}).  

Let $\N$ be the set of $n$-tuples $(a_1,\dots,a_{n})$ with 
$0\leq a_i \leq n+1-i$.  For $\n$ in $\N$, define 
$$\F_\n =\{(i,j) \mid i\leq j < i+a_i\}$$
and let $\C_\n$ be the full
subcategory consisting of direct sums of indecomposables 
$E^{ij}$ with $(i,j) \in \F_\n$.

\begin{proposition} The quotient-closed subcategories of $\rep \A$ are 
exactly those categories of the form $\C_\n$ for $\n \in \N$.  
\end{proposition}

In order to prove this proposition, we need a lemma:

\begin{lemma} \label{T_surjections}
Suppose $X$ is a representation of $\A$ which admits a surjection to 
$E^{kl}$.  Then any expression for $X$ as a direct sum of indecomposables
includes an indecomposable $E^{kj}$ with $j\geq l$. 
\end{lemma}

(Note that the statement of the lemma avoids assuming any uniqueness of 
the decomposition of $X$ as a sum of indecomposable representations.  
In fact, for $X \in \rep Q$, and $Q$ any quiver, the collection of 
summands appearing in an expression for $X$ as a sum of 
indecomposable representations is unique up to permutation.  This is called
the Krull-Schmidt property, and it is established, for example, in 
\cite[Section 1.4]{T_ASS}.  However, in the interests of self-containedness,
we have preferred to avoid the use of this.)

\begin{proof} 
Consider an expression of $X$ as a sum of indecomposables.  By 
Proposition \ref{T_morphisms}, $E^{kl}$ does not admit any morphisms from
$E^{ij}$ with $i<k$, so we may assume $X$ contains no such summands.  
On the other hand, when we consider summands of $X$ of the form
$E^{ij}$ with $i>k$, we see that the map from such summands to $E^{kl}$ cannot
be surjective at vertex $k$.  Therefore $X$ must have some summand of
the form $E^{kj}$ which admits a morphism to $E^{kl}$; using
Proposition \ref{T_morphisms} again, we see that $l\leq j$.  
\end{proof}

Now we prove the proposition:

\begin{proof} 
Clearly there are surjections: 
$$E^{ii}\twoheadleftarrow E^{i(i+1)} \twoheadleftarrow \dots 
\twoheadleftarrow E^{in}$$
It follows that a quotient-closed subcategory which contains 
$E^{ij}$ necessarily contains $E^{ip}$ for all $i\leq p \leq j$, and 
thus that any quotient-closed subcategory is of the form 
$\C_\n$ for some $\n\in \N$.  

Next, we verify that any such subcategory is quotient-closed. Suppose 
$X\in \C_\n$, and $X$ admits a surjection to $Y$.  Assume for the sake
of contradiction that $Y$ is not in $\C_\n$.  So $Y$ has some indecomposable
summand $E^{ij}$ which is not in $\C_\n$, and $E^{ij}$, in particular, admits a
surjection from $X$.  By Lemma \ref{T_surjections}, it follows that
$X$ has some summand of the form $E^{ik}$ with $k\geq j$, so $(i,k)\in \F_\n$,
so $(i,j)\in \F_\n$, 
contradicting the assumption that $Y\not\in \C_\n$.  
\end{proof}

\section{Subcategories ordered by inclusion}

We consider the obvious order on $\N$, the order it inherits as a Cartesian
product, and we write that $\n=(a_1,\dots,a_n) \leq \m=(b_1,\dots,b_n)$ iff
$a_i\leq b_i$ for all $i$.  

\begin{proposition} For $\n,\m \in \N$, $\C_{\n} \subseteq \C_{\m}$ 
iff $\n \leq \m$.  
\end{proposition}

\begin{proof} Clearly if $\n \leq \m$, then $\F_\n \subset \F_\m$, and therefore
$\C_{\n} \subseteq \C_{\m}$.  Conversely, if $\C_\n \subseteq \C_\m$, 
then in particular the indecomposable objects of $\C_\n$ are contained
among those of $\C_\m$.  Since the objects of $\C_\n$ are direct sums
of objects $E^{ij}$ with $(i,j)\in \F_\n$, the indecomposable objects
of $\C_\n$ are exactly those $E^{ij}$ with $(i,j)\in \F_\n$.  
It follows that $\F_\n\subseteq \F_\m$, and thus
$\n \leq \m$.  
\end{proof}

\section{Torsion classes in $\rep \A$}

We now define a subset of $\N$.  We say
that an $n$-tuple $(a_1,\dots,a_n)\in \N$ is a 
{\it bracket vector} if, for all
$1\leq i \leq n$ and $j\leq a_i$, 
we have that $j+a_{i+j}\leq a_i$.  
The well-formed bracket strings of length $2n+2$ correspond bijectively
to bracket vectors of length $n$: for each open-parenthesis, find 
the corresponding close-parenthesis, and then record the number of 
open-parentheses strictly between them.  Reading these numbers from
left to right, and skipping the last one (which is necessarily zero), we
obtain a bracket vector. 
Thus, for example, $()(())$ is encoded by the bracket vector $01$, while
$(()())$ is encoded by the bracket vector $20$.  
The notion of bracket vector goes back to Huang and Tamari \cite{T_HT}.
They show in addition that the poset structure induced on bracket
vectors from their inclusion into $\N$ is isomorphic to the Tamari 
lattice.  

The main result of this section is the following theorem:

\begin{theorem} \label{T_torsion}
The torsion classes of $\rep \A$ are exactly
the subcategories $\C_\n$ for $\n$ a bracket
vector.
\end{theorem}

Before we begin the proof of this theorem, we will first state
and prove the corollary which is the main result of this chapter.

\begin{corollary}
The torsion classes in $\rep \A$, ordered by inclusion, form a poset
isomorphic to the Tamari lattice.
\end{corollary}

\begin{proof} We have already observed that $\C_\n \subseteq \C_\m$
iff $\n \leq \m$.  It follows that the torsion classes for 
$\rep \A$, ordered by inclusion, form a poset isomorphic to the
poset structure induced on bracket vectors from their inclusion
into $\N$, which, as we have already remarked, is shown in \cite{T_HT} to be
 isomorphic to the Tamari lattice.  
\end{proof}

Next, we need to establish
some terminology and prove a lemma.  

For $\n$ a bracket vector, let 
$$\G_\n= \{(i,i+a_i-1)\mid 1\leq i \leq n, a_i\geq 1\}.$$

Let $\D_\n$ consist of the full subcategory consisting of sums of 
$E^{ij}$ for $(i,j) \in \G_\n$.  
Observe that $\D_\n$ is a subcategory of $\C_\n$, and 
any object in $\C_\n$ is a quotient of some 
object in $\D_\n$.

\begin{lemma} \label{T_trivial} Let $\n$ be a bracket vector.  
If $X\in \C_\n$, and $Z \in \D_\n$, then any extension
of $Z$ by $X$ is trivial.  \end{lemma}

\begin{proof} Write $Z=Z^1\oplus\dots\oplus Z^m$.  Observe that any
extension of $Z$ by $X$ can be realized by first forming an
extension of $Z^1$ by $X$, call it $Y^1$, then forming an extension of $Z^2$ 
by $Y^1$, call it $Y^2$,  and so on.  If the extension at each step is 
trivial, then the total extension is trivial, so it suffices to consider
the case that 
$Z$ is indecomposable. Suppose therefore that $Z\cong E^{i(i+a_i-1)}$, and
let $Y$ be an extension of $Z$ by $X$, for some $X\in \C_\n$. 

Let $t$ be an element of $Y_i$ which maps to a nonzero generator of 
$Z$.  Let $T$ be the subrepresentation of $Y$ generated by $t$.  
Let $v$ be the image of $t$ in $Y_{i+a_i}$.  

If $v$ is non-zero then, since $Z$ is not supported over 
$i+a_i$, we must have that $v\in X_{i+a_i}$.  Let $E^{kl}$ be a summand of 
$X$ in which $v$ is non-zero.  So we know that $k\leq i+a_i\leq l$.  
By the assumption that $(k,l)\in \F_\n$, it follows that $k\leq i$.  
Therefore, it follows
that, as in the proof of Proposition \ref{T_extensions}, we can find an
element $x$ of $X_i$ whose image in $X_{i+a_i}$ equals $v$.  Now
let $t'=t-x$.  The image of $t'$ in $Y_{i+a_i}$ is zero, so $T$ is
isomorphic to $Z$.  Therefore, the 
projection from $Y$ to $Z$ splits the inclusion of $T$ into $Y$, and the
extension of $Z$ by $X$ is trivial.  

\end{proof}

\begin
{proof}[Proof of Theorem \ref{T_torsion}]
First, we show that if 
$\n=(a_1,\dots,a_n) \in \N$ 
is not a bracket vector, then
$\C_\n$ is not a torsion class.  
So suppose we have some $i,j$ such that
$1\leq i \leq n$, $j\leq a_i$, and 
$j+a_{i+j}>a_i$.  
We know that $\C_\n$ contains
$E^{(i,i+a_i-1)}$ and $E^{(i+j,i+j+a_{i+j}-1)}$,
and from our assumptions, 
$i<i+j \leq i+a_i < i+j+a_{i+j}$.  
By Proposition \ref{T_extensions}, it follows
that $E^{i(i+j+a_{i+j}-1)} \oplus 
E^{(i+j,i+a_i-1)}$ is an extension of 
$E^{(i+j,i+j+a_{i+j}-1)}$ by $E^{(i,i+a_i-1)}$, and since $i+j+a_{i+j}>i+a_i$, we know
that $E^{(i+j,i+j+a_{i+j}-1)}$ is not contained in $\C_\n$.  Thus $\C_\n$ is not
closed under extensions, so it is not a torsion class.  

Now, we show that if $\n$ is a bracket vector, then $\C_\n$ is a torsion
class.  We have already shown that $\C_\n$ is quotient-closed, so all that
remains is to show that it is closed under extensions.  

Let $X$ and $Z$ be representations in $\C_\n$.  If we could assume
that $X$ and $Z$ were indecomposable, our lives would be much easier ---
an argument very similar to the converse direction would suffice.  However,
there is no reason that we can assume that.  

Choose an object $Z'\in \D_n$ such that 
$Z'$ has a surjection onto $Z$. 
Let $Y'$ be the pullback along
$Z'\rightarrow Z$ of the extension of $Z$ by $X$.  By Lemma \ref{T_trivial},
this is a trivial extension, so $Y' \in \C_\n$.  
By Lemma
\ref{T_pullback}, $Y'$ admits a surjective map to $Y$.  Thus $Y$ is a 
quotient of an element of $\C_\n$, and thus lies in $\C_\n$.  Therefore,
$\C_\n$ is closed under extensions.  
\end{proof}

\section{Related posets}
As was already mentioned, for arbitrary $Q$, we obtain a poset of 
torsion classes ordered by inclusion.  In fact, 
it is easy to check from the definition 
that the intersection of an arbitrary set of
torsion classes is again a torsion class.  Thus, this poset is closed
under arbitrary meets, and it has a maximum element, so it is a lattice.  

The number of indecomposable representations of $Q$ is finite if and only
if $Q$ is an orientation of a simply-laced Dynkin diagram.  (This is 
part of the celebrated theorem of Gabriel, see, for example,
\cite[Theorem VII.5.10]{T_ASS}.)
For such $Q$ (and only such $Q$),
the lattice of torsion classes is a (finite) 
Cambrian lattice, for a Coxeter element
chosen based on the orientation of $Q$.  See \cite{T_re} for more on Cambrian lattices, and \cite{T_IT,T_AIRT} for this result.  

\bigskip

The poset of torsion classes has not been classically studied in 
representation theory.  However, a closely related poset does appear.
For the remainder of this section, suppose that $Q$ is a quiver with
no oriented cycles.  

A representation $T$ of $Q$ is called a \emph{tilting object} if 
the only extension of
$T$ with itself is the trivial extension, and $T$ has $n$ pairwise
non-isomorphic summands ($n$ being the number of vertices of $Q$).  For $X\in \rep Q$, write
$\Gen X$ for the subcategory of $\rep Q$ consisting of all quotients of direct sums of 
copies of $X$.  A poset was defined by Riedtmann and Schofield \cite{T_RS} 
on the tilting objects of $\rep Q$, by
$T\geq V$ iff $\Gen T \supseteq \Gen V$.  This poset was studied further by Happel
and Unger \cite{T_HU2,T_HU1,T_HU3}.

This poset structure is related to the one discussed in this paper, because
if $T$ is a tilting object, then $\Gen T$ is a torsion class.  Further,
the torsion classes arising in this way can be described: they are just
the torsion classes which include all the injective representations of 
$Q$, see \cite[Theorem VI.6.5]{T_ASS}.  Thus, the torsion classes arising
in this way form an interval in the poset of all torsion classes,
whose minimal element is the torsion class consisting only of injective
representations, and whose maximal element is the torsion class consisting
of all representations.  

In the case of $\rep \A$, $\C_\n$ is sincere iff $a_1=n$, since the injective
indecomposable representations are those of the form $E^{1j}$ for 
$1\leq j \leq n$.  
There is a bijection from bracket vectors with $a_1=n$ to bracket
vectors of length $n-1$, by removing $a_1$. (A bracketing 
corresponding to a bracket vector with $a_1=n$ has its first open
parenthesis closed by the final close parenthesis of the bracketing,
from which this claim follows immediately.)
It therefore follows that the
Riedtmann-Schofield order on tilting objects for $\rep \A$  is 
isomorphic to the Tamari lattice $T_{n-1}$.   The poset structure on
tilting objects in the $A_n$ case 
was first analyzed in \cite{T_BK}.

%




\section*{Acknowledgements}

The author was partially supported by an NSERC Discovery Grant.  He thanks
the editors of the Tamari Festschrift for the invitation to participate
in their volume, and for helpful editorial suggestions.

\end{document}